\documentclass[12pt]{article}

\usepackage{amsmath,amssymb,amsthm,amsfonts,verbatim}
\usepackage{hyperref}
\usepackage{color}
\usepackage{epsfig}
\usepackage{enumerate}
\usepackage[all,2cell]{xy}
\CompileMatrices
\usepackage{microtype}
\usepackage[margin=1in]{geometry}

\theoremstyle{plain}
\newtheorem{theorem}{Theorem}
\newtheorem{lemma}[theorem]{Lemma}
\newtheorem{proposition}[theorem]{Proposition}
\newtheorem{corollary}[theorem]{Corollary}
\theoremstyle{definition}
\newtheorem{definition}[theorem]{Definition}

\newtheorem{question}[theorem]{Question}

\newcommand{\Z}{\mathbb{Z}}

\newcommand\T{\mathcal{T}}

\newcommand{\ov}{\overline}

\DeclareMathOperator{\im}{im}
\DeclareMathOperator{\Aut}{Aut}
\DeclareMathOperator{\Div}{Div}
\DeclareMathOperator{\Pic}{Pic}
\newcommand{\corr}{\leftrightsquigarrow}

\renewcommand{\epsilon}{\varepsilon}
\newcommand{\coloneq}{\mathrel{\mathop:}\mkern-1.2mu=}

\newcommand{\para}[1]{\medskip\noindent\textbf{#1.}}

\newcommand{\mathqedhereforthm}{}

\title{Rotor-routing and spanning trees on planar graphs}

\author{Melody Chan\thanks{The first and second authors gratefully acknowledge support from the National Science Foundation.}, \,Thomas Church\footnotemark[1], \,and Joshua A. Grochow}

\begin{document}
\maketitle

\begin{abstract}
The sandpile group $\Pic^0(G)$ of a finite graph $G$ is a discrete analogue of the Jacobian of a Riemann surface which was rediscovered several times in the contexts of arithmetic geometry, self-organized criticality, random walks, and algorithms. Given a ribbon graph $G$, Holroyd \emph{et al.} 
used the ``rotor-routing'' model to define a free and transitive action of $\Pic^0(G)$ on the set of spanning trees of $G$. However, their construction depends \emph{a priori} on a choice of basepoint vertex. Ellenberg asked whether this action does in fact depend on the choice of basepoint. We answer this question by proving that the action of $\Pic^0(G)$ is independent of the basepoint if and only if $G$ is a planar ribbon graph.
\end{abstract}

\section{Introduction}
The abelian sandpile model and rotor-routing model are combinatorial models of various dynamics on graphs that were rediscovered several times in several different communities, ranging from combinatorics to self-organized criticality, arithmetic geometry, and algorithms. (For example, \cite{dhar,PDDK}; see \cite{rotor} for more details and references.) The stable configurations in the abelian sandpile model on a graph $G$ form a group called the {sandpile group} or the {critical group} of $G$; it is also known as the {Picard group} of the graph, which we will denote $\Pic^0(G)$ in analogy with the Picard group of a Riemann surface.   

It has long been known that the order of $\Pic^0(G)$ is equal to the number of spanning trees of $G$ (see, e.\,g., \cite{biggs} and references therein).  Indeed, this is a form of Kirchoff's Matrix--Tree Theorem.  A natural first question to ask, then, is whether there is a canonical bijection between the elements of $\Pic^0(G)$ and the spanning trees of $G$. 

What does ``canonical'' mean here? At the very least, we should require that the bijection be invariant under automorphisms of $G$, for certainly it shouldn't depend on what we name the vertices and edges. But there is clearly no such bijection: many graphs, e.\,g.~a complete graph or an $n$-cycle, don't even have a distinguished spanning tree that could correspond to the identity element of $\Pic^0(G)$.  

So it is natural to ask instead for the next best thing: \emph{is there a canonical free, transitive action of the sandpile group on the set of spanning trees?}  In other words, is there a canonical $\Pic^0(G)$-torsor structure on the set of spanning trees of $G$?  By ``canonical,'' we still mean an action that is invariant under automorphisms of $G$.  Such an action would give a bijection between the sandpile group and the set of spanning trees, once we decide on a spanning tree that corresponds to the identity element.

But this is still too much to ask for. For example, consider the graph  $G$ with two vertices and $n$ parallel edges between them. The sandpile group $\Pic^0(G)$ is cyclic of order $n$, and each edge of $G$ is itself a spanning tree. So an automorphism-invariant action of $\Pic^0(G) = \Z/n\Z$ amounts to an automorphism-invariant cyclic ordering of the edges, which clearly does not exist. 

This example, and others like it, suggest that we might hope for a free transitive action that is canonically defined on $G$ once we've fixed the additional structure of a cyclic ordering of the edges incident to each vertex. A graph together with such a system of cyclic orderings is called a \emph{ribbon graph} (or sometimes ``combinatorial embedding''). We therefore arrive at the following question.

\begin{question}
\label{q:intro}
Is there a free and transitive action of $\Pic^0(G)$ on the spanning trees of $G$ that is invariant under automorphisms preserving a ribbon graph structure on $G$?
\end{question}
Holroyd \emph{et al.} \cite{rotor} gave one possible way to answer Question~\ref{q:intro}. Given a ribbon graph \emph{and} a choice of a basepoint vertex, they used rotor-routing to define a free and transitive action of $\Pic^0(G)$ on the spanning trees of $G$. This action is explained in Section~2.  Therefore an affirmative answer to Question~\ref{q:intro} would follow if one could prove that the basepoint in the construction of Holroyd \emph{et al} is unnecessary. In other words, we have the following question, which to our knowledge was first asked by Ellenberg \cite{MO}: does the rotor-routing action of $\Pic^0(G)$ on the spanning trees of $G$ depend on the choice of basepoint?

In this paper we prove that if $G$ is a \emph{planar} ribbon graph, then miraculously, the action defined by rotor routing is independent of the basepoint and hence canonical. Furthermore, we show that this characterization is tight: only planar ribbon graphs have this property. Our main theorem is thus:

\begin{theorem}
\label{theorem:main}
Let $G$ be a connected ribbon graph. The action of the sandpile group $\Pic^0(G)$ on the set $\T(G)$ of spanning trees is independent of the choice of basepoint vertex if and only if $G$ is a planar ribbon graph.
\end{theorem}

The proof is based on three key ideas. First, the rotor-routing action of the sandpile group on spanning trees can be partially modeled via \emph{rotor-routing on unicycles} (\cite[\S 3]{rotor}).
This is a related dynamical system with the property that rotor-routing becomes periodic, rather than terminating after finitely many steps. Rotor-routing on unicycles is described in Section~\ref{sec:reversibility}.

The second main idea is that the independence of the sandpile action on spanning trees can be described in terms of \emph{reversibility of cycles}. We introduce the notion of reversibility (previously considered in \cite{rotor} only for planar graphs), and prove in Proposition~\ref{prop:reversibility2} that reversibility is a well-defined property of cycles in a ribbon graph. We establish a relation between reversibility and basepoint-independence in Section~\ref{sec:spanningtrees}. 

Third, reversibility is closely related to whether a cycle separates the surface corresponding to the ribbon graph into two components. We prove in Lemma~\ref{lemma:almostseparating} that these conditions are almost equivalent. Moreover, although they are not equivalent for individual cycles, we prove in Proposition~\ref{prop:planar_is_reversible} that \emph{all} cycles are reversible if and only if all cycles are separating, in which case the ribbon graph is planar.


\section{Definitions and notation} 
By a \emph{graph} we mean a connected and undirected graph, with multiple edges allowed but no self-loops. We write $V(G)$ and $E(G)$ for the vertex set and edge set of $G$, respectively. Throughout, $G$ denotes a \emph{ribbon graph}: a graph together with a cyclic ordering of the edges incident to each vertex. Ribbon graphs are equivalent to graphs embedded into orientable surfaces (hence their alternate name ``combinatorial embeddings''): from any ribbon graph one can reconstruct a closed orientable surface in which it is naturally embedded, and conversely the local orientation of the surface at a vertex gives the cyclic ordering (see, e.\,g., \cite[Theorem~3.7]{thomassen}). We denote by $\T(G)$  the set of spanning trees of $G$. When the graph is clear from context, we sometimes write $\T$ for $\T(G)$.

\para{The rotor-routing model} A \emph{rotor configuration} $\rho$ on a directed graph is an assignment to each vertex $y$ of an outgoing edge $\rho[y]$ based at $y$. When we speak about rotor-routing or rotor configurations on a (ribbon) graph $G$, we will always mean rotor configurations on the underlying Eulerian directed graph $\underline{G}$ obtained by replacing each undirected edge of $G$ with two oppositely oriented edges.  

Fix a ribbon graph $G$.  The \emph{rotor-routing model} is a deterministic process on pairs $(\rho,x)$ consisting of a rotor configuration $\rho$ and a vertex $x\in V(G)$. We think of $x$ as the position of a \emph{chip} that will move around the graph as the process is iterated. One step of the rotor-routing process consists of replacing $(\rho,x)$ by a new pair $(\sigma,y)$ defined as follows. First, the rotor configuration $\sigma$ coincides with $\rho$ at every vertex except at $x$, where $\sigma[x]$ is the edge leaving $x$ that follows $\rho[x]$ in the cyclic order. Second, the new position of the chip $y$ is the vertex at the other end of the edge $\sigma[x]$ from $x$. In short, the rotor-routing process rotates the rotor $\rho[x]$ at the current position of the chip to the next position, and then moves the chip along this edge in the new direction of the rotor. In this paper we will frequently iterate the rotor-routing process until a certain ``stopping condition'' is reached, and reason about the resulting configuration and the steps taken to reach it.


\para{The sandpile group} The group $\Div(G)$ of \emph{divisors} on the graph $G$ is the free abelian group of formal $\Z$-linear combinations $\sum_{v \in V(G)} n_v v$ of the vertices of $G$. In this paper we will be mostly interested in the subgroup $\Div^0(G)$ of ``degree-0'' divisors:
\[\Div^0(G)=\left\{\left.\sum_{v \in V(G)} n_v v~\right|~n_v\in \Z, \, \sum_v n_v=0\right\}.\]
If we fix a vertex $r$ (for ``root''), $\Div^0(G)$ is freely generated by the basis $\{v-r\}_{v\neq r}$ as $v$ ranges over the other vertices of $G$.

The \emph{graph Laplacian} of a graph $G$ is an operator $\Delta\colon \Div(G)\to \Div(G)$ closely related to the adjacency matrix of $G$. With respect to the basis of vertices $v$, the operator $\Delta$ corresponds to the symmetric matrix whose diagonal entries are $\Delta_{vv}$ = degree of $v$ and whose off-diagonal entries $\Delta_{vw} = (-1)\cdot \# \{\text{edges } v \leftrightarrow w\}$. The sum $\Delta_{vv}+\sum_w\Delta_{vw}$ of each row and column is 0, which shows both that the image of $\Delta$ is contained in $\Div^0(G)$, and that $\Delta(\sum_v v)=0$. 
The image of $\Delta$ is a finite-index subgroup of $\Div^0(G)$, and the \emph{sandpile group} $\Pic^0(G)$ is the quotient:
\[\Pic^0(G) \coloneq \Div^0(G) / \im(\Delta).\]
For any graph $G$ the group $\Pic^0(G)$ is a finite abelian group whose order is, up to sign, the determinant $\det(\Delta)$ of any $(n-1)\times(n-1)$ submatrix of the graph Laplacian.

People who study sandpile groups usually prefer to work with the distinguished representatives of each divisor class that are recurrent under a certain dynamical process.  We will not describe the beautiful story of sandpile dynamics here, but that viewpoint does implicitly play a key role in the rotor-routing action that we now define.


\para{Action of $\Pic^0(G)$ on $\T(G)$} Given a choice of root $r \in V(G)$, Holroyd \emph{et al.} \cite{rotor} defined a natural action of $\Div^0(G)$ on $\T(G)$, which we now describe. This action is trivial on any divisor in $\im(\Delta)$ and thus descends to an action of $\Pic^0(G)$ on $\T(G)$.

\begin{figure}%
\centering
\includegraphics[height=2in]{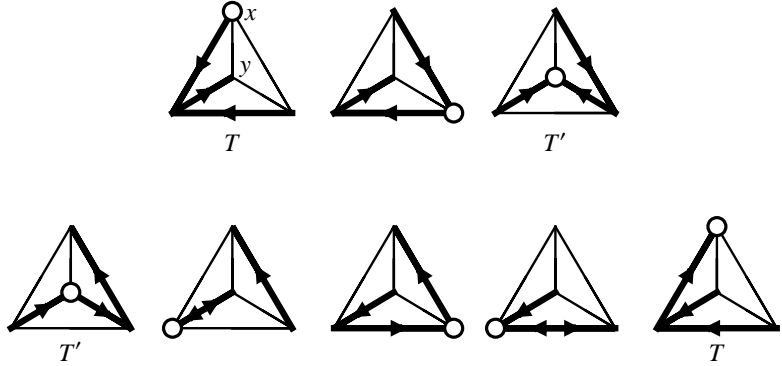}%
\caption{The rotor-routing process in the top row shows that $(x-y)_y(T)=T'$.  The rotor-routing process in the bottom row shows that $(y-x)_x(T')=T$ and thus $(x-y)_x(T)=T'.$  So the action of $x-y$ on $T$ is independent of our choice of $x$ or $y$ as basepoint.  Here, the white circle indicates the position of the chip, and the rotors are all oriented clockwise relative to the page.}%
\label{f:planar}%
\end{figure}

For a fixed root $r$, given a spanning tree $T\in \T(G)$ and a divisor $D\in \Div^0(G)$, we will write $D_r(T)$ for the image of $T$ under the rotor-routing action (based at $r$). 
We define the action of $\Div^0(G)$ by describing the action $(v-r)_r$ of the generators $v - r$ for each  $v\in V(G)$, as follows. Given a spanning tree $T$, orient the edges of $T$ towards $r$. The resulting collection of directed edges defines a rotor configuration on $G$, except that there is no rotor at the root $r$; we think of this as a rotor configuration $\rho$ on the graph obtained from
$\underline{G}$ by removing the edges leaving $r$. Place a chip on the vertex $x$, and iterate the rotor-routing process starting with $(\rho,x)$. Stop when the chip first reaches $r$  (which it necessarily must \cite[Lemma~3.6]{rotor}), and call the resulting state $(\sigma,r)$. The rotor configuration $\sigma$ is in fact another spanning tree $T'$ oriented towards $r$, and we define $(x-r)_r(T)$ to be the spanning tree $T'$. See Figure~\ref{f:planar}. \pagebreak

Since $\Div^0(G)$ is an abelian group, the action of different divisors $D,D'\in \Div^0(G)$ will commute: $D_r D'_r = D'_r D_r$. However, we emphasize that the actions corresponding to \emph{different} root vertices $r$ and $s$ may \emph{not} commute, so it is important to distinguish between $D_r(D'_s(T))$ and $D'_s(D_r(T))$. 

The action of $\Div^0(G)$ descends to the sandpile group $\Pic^0(G)$. In fact, for any choice of root vertex $r$, the resulting action of $\Pic^0(G)$ on $\T(G)$ is free and transitive \cite[Lemmas 3.17 and 3.19]{rotor}. (This property has been vastly generalized to abelian networks by Bond and Levine in their recent preprint \cite{BL}.) 

We can now restate Ellenberg's question \cite{MO} as: if $D_r\colon \T(G)\to \T(G)$ is the permutation defined by this action, does $D_r=D_s$ for all $r,s\in V(G)$ and $D\in\Div^0(G)$?  In other words, is the rotor-routing $\Pic^0(G)$-torsor structure on the set of spanning trees $\T(G)$ independent of basepoint? Theorem~\ref{theorem:main} completely answers this question.

\section{Reversibility and separating cycles}
\label{sec:reversibility}
A \emph{path} $P$ in the graph $G$ is an edge-injective map $P_k\to G$ from the length-$k$ path $P_k$; it may or may not be vertex-injective. 
A \emph{cycle} $C$ in $G$ is the image of a vertex-injective and edge-injective map $C_k\to {G}$ from the length-$k$ cycle $C_k$ (a circle subdivided into $k$ edges). We consider both paths $P$ and cycles $C$ to be oriented, but we identify maps $C_k\to G$ that differ by cyclic permutation (in other words, we do not specify a basepoint for cycles). We denote by $\ov{P}$ or $\ov{C}$ the same path or cycle with the opposite orientation.

Given an oriented cycle $C$ in a ribbon graph and a vertex $x \in C$, the edges incident to $x$ naturally fall into three classes: the two edges involved in $C$, the edges (if any) on the left of $C$, and the edges (if any) on the right of $C$. The edges to the right of $C$, for example, are the ones that occur after the out-edge at $x$ in $C$ and before the in-edge at $x$ in $C$, in the cyclic order at $x$. These three classes are disjoint, since our graphs have no self-loops; however, if an edge $e$ has both endpoints $x$ and $y$ on $C$, it is possible that $e$ is on the left of $C$ at $x$ but on the right of $C$ at $y$.

A cycle $C$ is \emph{nonseparating} if there exists a path $P$ in $G$ whose endpoints lie on $C$ and is disjoint from $C$ otherwise, with its first edge on the left of $C$ and its last edge on the right of $C$. Such a path $P$ is a \emph{witness} that $C$ is nonseparating. The cycle $C$ is \emph{separating} if no such witness exists; this is equivalent to saying that $C$ is a separating curve on the surface associated to the ribbon graph $G$. The ribbon graph $G$ is \emph{planar} if every cycle is separating (equivalently, if its associated surface is a sphere). 

\para{Unicycles and rotor-routing}
A \emph{unicycle} $(\rho,v)$ consists of a rotor configuration $\rho$ which contains exactly one directed cycle $C(\rho)$, together with a vertex $v$ lying on the cycle $C(\rho)$.  It is not difficult to see that applying rotor-routing to the configuration $(\rho,v)$ preserves these conditions, and thus takes unicycles to unicycles \cite[Lemma 3.4]{rotor}. 
In fact, the following is is a key observation that we will use many times:

\begin{lemma}[{\cite[Lemma 4.9]{rotor}}]
\label{lemma:periodic}
Let $(\rho,v)$ be a unicycle on a graph with $m$ edges.  Iterating the rotor routing process $2m$ times starting at $(\rho,v)$, the chip traverses each edge of $G$ exactly once in each direction, each rotor makes exactly one full turn, and the stopping state is $(\rho, v)$.
\end{lemma}

We write $(\rho,x)\corr (\sigma,y)$ if the configuration $(\sigma,y)$ can be obtained from $(\rho,x)$ by iterating the rotor-routing process some positive number of times. By 
Lemma~\ref{lemma:periodic},
the relation $(\rho,x)\corr (\sigma,y)$ is in fact an equivalence relation on unicycles.  
The following lemma shows that each equivalence class is naturally in bijection with the directed edges of $G$.
We write $(\rho,x)\leadsto_{y,e}(\sigma,y)$ if $(\rho,x)\corr(\sigma,y)$ and $\sigma[y]=e$.
\begin{lemma}
\label{lemma:therotorknows}
Let $(\rho,x)$ be a unicycle. For any vertex $y$ and any directed edge $e$ based at $y$, there exists a unique rotor configuration $\sigma$ such that $(\rho,x)\leadsto_{y,e} (\sigma,y)$.
\end{lemma}
\begin{proof}
Run the rotor-routing process starting at $(\rho,x)$ until the next occurrence of $(\rho,x)$.  According to 
Lemma~\ref{lemma:periodic}, each rotor makes precisely one full rotation in this process. Since the rotor at $y$ only advances when the chip is at $y$, this implies that at some intermediate stage $(\sigma,y)$ we had $\sigma[y]=e$. The next rotor-routing step then advances the rotor at $y$, so $\sigma$ is unique.
\end{proof}
The \emph{first} time the chip arrives at $y$ when rotor-routing is iterated starting with $(\rho,x)$, the current rotor configuration $\sigma$ satisfies $\sigma[y]=\rho[y]$, since $y$ has not previously been visited. Accordingly we  write $(\rho,x)\leadsto_y (\sigma,y)$ as an abbreviation for $(\rho,x)\leadsto_{y,\rho[y]}(\sigma,y)$. Finally, when $(\sigma,y)$ has already been defined, we write $(\rho,x)\leadsto (\sigma,y)$ as shorthand for the rotor routing process beginning with $(\rho,x)$ and ending at the first occurrence of $(\sigma, y)$.

\begin{definition} \label{def:reversible}
Given a unicycle $(\rho,v)$ with directed cycle $C=C(\rho)$, let $\ov{\rho}$ denote the configuration obtained from $\rho$ by reversing the edges of $C$, and keeping all other rotors unchanged.  We say that $C$ is \emph{reversible} if $(\rho, v) \corr (\ov{\rho}, v)$.
\end{definition}
\noindent Before using this terminology, we need to verify that it is well-defined: in other words, that reversibility really is a property of the \emph{cycle} $C$, and does not depend on the choice of $\rho$ nor the choice of $v$. Note that by definition, $C$ is reversible if and only if $\ov{C}$ is reversible. 

We will prove in Proposition~\ref{prop:reversibility2} that reversibility is well-defined, but we first need the following technical result. This result will also be used, along with Lemma~\ref{lemma:almostseparating}, to show that reversible cycles are quite close to being separating.
\begin{proposition}
\label{prop:LCRC}
Let $C$ be a directed cycle. If there exists a unicycle  $(\rho,v)$  with $C=C(\rho)$ such that $(\rho,v)\corr (\ov{\rho},v)$, then the vertices $y\not\in C$ can be partitioned into two sets $L_C\sqcup R_C$ such that:
\begin{enumerate}
\item if $y\not\in C$ is adjacent to $x\in C$ along an edge lying on the left of $C$ at $x$, then $y\in L_C$.
\item if $y\not\in C$ is adjacent to $x\in C$ along an edge lying on the right of $C$ at $x$, then $y\in R_C$.
\item if $y\not\in C$ is adjacent to $z\not\in C$, then  $y$ and $z$ are either both in $L_C$ or both in $R_C$.\qedhere
\end{enumerate}
\end{proposition}

Any separating curve $C$ separates the surface into two parts $L_C$ and $R_C$ as in this proposition. The key utility of Proposition~\ref{prop:LCRC} is that it allows us to define the sets $L_C$ and $R_C$ knowing only that $C$ is \emph{reversible}, without knowing whether or not $C$ is separating. 

\begin{proof}
For any vertex $y$ in $G$, let $d_y$ be its degree. Consider the rotor-routing process that takes $(\rho,v)$ to $(\ov{\rho},v)$.
For any vertex $y$ \emph{not} lying on $C$, the rotor ends at its initial position $\rho[y]=\ov{\rho}[y]$, so by Lemma~\ref{lemma:periodic}, the vertex $y$ was visited either 0 or $d_y$ times. We define $L_C=\{y\not\in C|\text{$y$ is visited 0 times}\}$ and $R_C=\{y\not\in C|\text{$y$ is visited $d_y$ times}\}$. Since $G$ has no isolated vertices (by our convention, all graphs are connected), these two sets are disjoint. Since no directed edge is traversed more than once, in the latter case $y\in R_C$ the chip must arrive at $y$ along each of the $d_y$ adjacent edges exactly once and leave $y$ along each edge exactly once.

Now consider a vertex $x$ lying on $C$. Since the cycle $C$ has been reversed, the rotor at $x$ begins at $\rho[x]=C[x]$ and ends at $\ov{\rho}[x]=\ov{C}[x]$. Therefore the rotor rotates $d_x^R$ times, where $d^R_x$ is 1 plus the number of edges at $x$ on the \emph{right} of $C$, and so the chip visits $x$ precisely $d_x^R$ times. The first $d_x^R-1$ times the chip leaves $x$, it leaves along the $d_x^R-1$ edges lying on the right of $C$; the final time it leaves $x$, it leaves along the edge $\ov{\rho}[x]=\ov{C}[x]$ of the reversed cycle $C$ (since this is the final state of the rotor).

If $y\not\in C$ is adjacent to $x\in C$ along an edge lying on the right of $C$ at $x$, the chip will leave $x$ along this edge at some point, so $y\in R_C$. Similarly, if $y\not\in C$ is adjacent to $x\in C$ along an edge lying on the left of $C$ at $x$, the chip does not leave $x$ along this edge; in other words, the chip does not arrive at $y$ along this edge. This implies that $y\in L_C$, because we noted above a vertex $y$ lying in $R_C$ must be visited along every one of its incoming edges, and this would include the edge from $x$. Finally, if $y\not\in C$ is adjacent to $z\not\in C$ and the chip traverses this edge in either direction, then both $y$ and $z$ lie in $R_C$; if not, both lie in $L_C$.
\end{proof}

The three conditions above uniquely characterize $L_C$ and $R_C$ as those vertices that can be connected by a path to the left side or right side of $C$ respectively; since this description does not depend on $\rho$ or $v$, the sets $L_C$ and $R_C$ depend only on the cycle $C$. This seems at first to prove that $C$ is separating. However, it does not rule out the possibility of a witness consisting of a single edge  joining the left side of $x\in C$ to the right side of $x'\in C$. In fact, nonseparating cycles of this form can indeed be reversible, as we will see in the proof of Proposition~\ref{prop:planar_is_reversible}.

\begin{proposition}\label{prop:reversibility2}
Let $C$ be a directed cycle. If $(\rho',v')\corr (\ov{\rho}',v')$ for \emph{some} unicycle $(\rho',v')$ with $C=C(\rho')$, then 
$(\rho, v)\corr (\ov{\rho}, v)$ for \emph{any} unicycle $(\rho, v)$ with $C=C(\rho)$.
In other words, the reversibility of a cycle $C$, as in Definition~\ref{def:reversible}, is well-defined: it only depends on $C$, and not on a choice of unicycle.
\end{proposition}
\begin{proof}
Given a unicycle $(\rho,v)$ with $C=C(\rho)$, define its \emph{maximal reversal} $(\sigma,u)$ to be the first state $(\sigma,u)$ encountered in the rotor-routing process such that $u\in C$ and $\sigma[u]=\ov{C}[u]$. In other words, the chip is on the cycle $C$ at the vertex $u$, and the rotor at $u$ is about to rotate \emph{past} the reversed cycle $\ov{C}$ to the edges on the left of $C$ for the first time. We will prove that as long as $(\rho',v')\corr (\ov{\rho}',v')$ for \emph{some} $\rho'$ and some $v'\in C$,  the maximal reversal of $(\rho,v)$ is  $(\ov{\rho},v)$.  Then we will have $(\rho,v)\corr (\sigma,u)\! =\! (\ov{\rho},v)$ as desired. (We cannot use Lemma~\ref{lemma:therotorknows} to deduce this from $\sigma[u]=\ov{C}[u]=\ov{\rho}[u]$, because we do not yet know that $(\rho,u)\corr (\ov{\rho},u)$.) 

Assume that $(\rho',v')\corr(\ov{\rho}',v')$ for some $\rho'$ with $C=C(\rho')$, and let $L_C$ and $R_C$ be the sets defined by Proposition~\ref{prop:LCRC}.
The set of directed edges traversed in the process $(\rho',v')\leadsto (\ov{\rho}',v')$ was described in the proof of that proposition: the chip traverses (1) every directed edge leaving $y\in R_C$, (2) each directed edge leaving $x\in C$ on the right side of $C$, and (3) each edge $\ov{C}[x]$ leaving $x\in C$. This set of directed edges only depends on $C$, not on $\rho'$ or $v'$, so we will denote it by $E_C$.

For any vertex $w$, let $e_w$ denote the number of edges in $E_C$ leaving $w$. (Specifically,  $e_x=d_x^R$ for $x\in C$, $e_y=d_y$ for $y\in R_C$, and $e_y=0$ for $y\in L_C$.) Since the chip ends up back at $v$, it leaves each vertex $w$ the same number of times that it arrives there, so $e_w$ is also the number of directed edges in $E_C$ \emph{arriving} at $w$.

Now consider an arbitrary rotor configuration $(\rho,v)$ with $C=C(\rho)$ and $v\in C$. The key to this lemma is the observation that if $(\sigma,u)$ is the maximal reversal of $(\rho,v)$, the process $(\rho,v)\leadsto (\sigma,u)$ only traverses edges in $E_C$. Assume otherwise, and let $e$ be the first directed edge \emph{not} in $E_C$ traversed. Since the chip begins on $C$, and no edges connect $R_C$ to $L_C$, this first edge $e$ is based at some $x\in C$. By definition of $E_C$, either $e$ lies on the left side of $C$, or $e=C[x]$ itself. In either case, the rotor at $x$ must rotate past $\ov{C}[x]$ before it can reach the edge $e$. But this is a contradiction, since  by definition $(\sigma,u)$ is the first time the rotor at $x\in C$ is ${\overline C}[x]$ while the chip is there.

In the process $(\rho, v)\leadsto (\sigma, u)$, the rotor at $u\in C$ rotates from $\rho[u]=C[u]$ to $\sigma[u]=\ov{C}[u]$, so the chip leaves $u$ precisely $e_u$ times. However, if $u\neq v$, the chip must have arrived at $u$ precisely $e_u+1$ times, since it did not originate there. Since the total number of edges in $E_C$ arriving at $u$ is only $e_u$, this is a contradiction. This shows that $u=v$, and moreover that the chip has arrived at $v$ along each of the $e_{v}$ edges in $E_C$ directed towards $v$.

One of these is the edge $\ov{C}[w]$, where $w\in C$ is the next vertex after $v$ in the cycle $C$, so we must have $\sigma[w]=\ov{C}[w]$ as well. But if the rotor at $w$ has reversed from $\rho[w]=C[w]$ to $\sigma[w]=\ov{C}[w]$, the chip must have left $e_w$ times. To arrive $e_w$ times at $w$, the rotor at the \emph{next} vertex after $w$ must have been reversed, and so on. By induction, we conclude that  we have $\sigma[x]=\ov{C}[x]=\ov{\rho}[x]$ for all $x\in C$. In particular, the directed cycle of $\sigma$ is $C(\sigma)=\ov{C}$.

Since the process $(\rho,v)\leadsto (\sigma,v)$ does not visit $L_C$, we know that $\sigma[y]=\rho[y]=\ov{\rho}[y]$ for all $y\in L_C$. It remains to show that $\sigma[y]=\rho[y]$ for $y\in R_C$. To do this, let $(\tau,v)$ be the maximal reversal of $(\sigma,v)$.
Since $C(\sigma)=\ov{C}$, the process $(\sigma,v)\leadsto (\tau,v)$ does not visit $L_{\ov{C}}=R_C$, so $\tau[y]=\sigma[y]$ for all $y\in R_C$. But the previous paragraph shows that $\tau[v]=\overline{\sigma}[v]=\rho[v]$. Since $(\rho,v)\corr (\sigma,v)\corr (\tau,v)$, Lemma~\ref{lemma:therotorknows} implies that $\tau=\rho$. Therefore $\sigma[y]=\rho[y]$ for all $y\not\in C$, and $\sigma[x]=\ov{\rho}[x]$ for all $x\in C$, demonstrating that $\sigma=\ov{\rho}$ as desired.
\end{proof}

\begin{lemma}
\label{lemma:almostseparating}
Let $C$ be a directed cycle. If $C$ is reversible, then any path $P$ witnessing that $C$ is nonseparating must have length 1. Conversely, if $C$ is a separating cycle, then $C$ is reversible.
\end{lemma}
\begin{proof}
Let $P$ be a path of length $k\geq 2$, with vertices $p_0,p_1,\ldots,p_k$, with endpoints $p_0\in C$ and $p_k\in C$. If $P$ is a witness that $C$ is nonseparating,  the second vertex $p_1$  lies in $L_C$, and the second-to-last vertex $p_{k-1}$  lies in $R_C$. But this contradicts condition (3) of Proposition~\ref{prop:LCRC}, which says that membership in $L_C$ is locally constant along paths in $G-C$.  Therefore any witness $P$ must have length 1.

To prove that any separating cycle is reversible, we will go through the proof of Proposition~\ref{prop:LCRC} in reverse; the resulting argument closely parallels \cite[Corollary 4.11]{rotor}, where this was proved for planar ribbon graphs.  (In fact, a special case of reversibility, for cycles in an $n\times n$ planar grid graph, was proved even earlier, in \cite[Proposition I]{pps}.)

If $C$ is a separating cycle, we can partition the vertices $y\not\in C$ into $L'_C\sqcup R'_C$, where $y\in L'_C$ (resp.\ $y\in R'_C$) if there exists a path from $y$ to $C$ ending with an edge on the left side (resp.\ right side) of $C$. Every $y$ can be connected to $C$ by some path,  since $G$ is connected, so $y\in L'_C\cup R'_C$. If $y\in L'_C\cap R'_C$ could be connected to both sides of $C$, splicing these paths together would give a witness that $C$ was nonseparating, so $L'_C\cap R'_C=\emptyset$.

Let $H$ be the ribbon graph obtained from $G$ by deleting all vertices in $R'_C$, all edges adjacent to $R'_C$, and all edges lying on the right of $C$. Because $C$ is separating, no edge on the left side of $C$ is removed (every such edge either joins $x\in C$ to $y\in L'_C$, or joins $x\in C$ to $x'\in C$, in which case it lies on the left side of $C$ at both endpoints).
Since no rotor in $\rho|_{C \cup L'_C}$ points along a removed edge, we can restrict the rotor configuration $\rho$ to  $H$ as $\rho|_H$.

It is easy to run the rotor routing process  starting at $(\rho|_H, v)$, because at each $x\in C$ the two edges of $C$ are now adjacent in the cyclic order. Thus at each step, the rotor at one vertex on $C$ is reversed, and the chip moves to the previous vertex on $C$. After a number of steps equal to the length of $C$,  each rotor on $C$ has moved once, the cycle $C$ has been reversed, and the chip is back at $v$; in other words, the process takes $(\rho|_H,v) \leadsto (\ov{\rho}|_H, v)$.

Since  rotor-routing  on $H$ is periodic, continuing  gives a rotor-routing process on $H$ taking $(\ov{\rho}|_H, v)$ to $(\rho|_H, v)$. But this segment of the process does not involve any of the deleted edges. Therefore if the rotor-routing process is run in parallel on $G$ starting with $(\ov{\rho},v)$, the two processes will take precisely the same series of steps. When the former terminates at $(\rho|_H,v)$, the latter will therefore be at $(\rho, v)$, demonstrating that $(\ov{\rho},v)\corr (\rho,v)$ as desired.
\end{proof}

%

\begin{proposition}\label{prop:planar_is_reversible}
A connected ribbon graph $G$ without loops is planar if and only if all cycles on $G$ are reversible.
\end{proposition}
\begin{proof}
If $G$ is planar, the Jordan Curve Theorem implies that every cycle is separating, so Lemma~\ref{lemma:almostseparating} implies that every cycle is reversible.  

For the converse, suppose $G$ is not planar, so it has at least one nonseparating cycle, but that every cycle on $G$ is reversible. Our first claim is that every nonseparating cycle on $G$ has length 2. Indeed,  let $P$ be a path witnessing that the cycle $C$ is nonseparating, joining the left side of $x\in C$ to the right side of  $x'\in C$. This path splits $C$ into two directed arcs from $x$ to $x'$. Call them  $A_1$ and $A_2$, labeled so that $C=A_1\sqcup \ov{A}_2$. Let $C_1$ be the cycle $A_2\sqcup \ov{P}$, and let $C_2$ be the cycle $P\sqcup \ov{A_1}$. The arc $A_1$ is now a witness that $C_1$ is nonseparating, since $A_1$ lies on the left of $C_1$ at $x$ but on the right of $C_1$ at $x'$.  Similarly $A_2$ witnesses that $C_2$ is nonseparating, since $A_2$ lies on the right of $C_2$ at $x$ but on the left of $C_2$ at $x'$. Since all three cycles $C$, $C_1$, and $C_2$ are reversible by assumption,
 Lemma~\ref{lemma:almostseparating}  implies that the paths $P$, $A_2$ and $A_1$ each have length 1. In particular, the original cycle $C=A_1\sqcup\ov{A_2}$ has length 2, as claimed. 

Let us rename the edge $P$ to $A_3$, and the cycle $C$ to $C_3$, to expose the cyclic symmetry of our notation. Then the cyclic order of these edges at $x$ and $x'$ is the same, namely $A_1,A_2,A_3$ and $\ov{A}_1,\ov{A}_2,\ov{A}_3$ respectively, and $C_i=A_{i+1}\ov{A}_{i+2}$ with indices modulo 3.
See Figure~\ref{fig:a1a2a3}.  Let us say that an edge at $x$ \emph{lies between} $A_i$ and $A_{i+1}$, for $i=1$, $2$, or $3$, if it comes after $A_i$ and before $A_{i+1}$ in the cyclic order of edges at $x$.  Similarly, we will say that an edge at $x'$ lies between $\ov{A_{i+1}}$ and $\ov{A_i}$ if it comes after $\ov{A_{i+1}}$ and before $\ov{A_i}$ in the cyclic ordering of edges at $x'$.  

\begin{figure}%
\begin{center}
\includegraphics[width=1.5in]{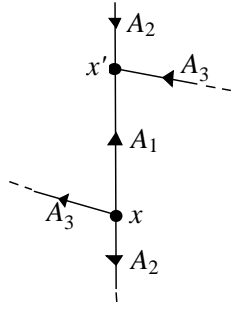}%
\end{center}
\caption{The edges $A_1, A_2,$ and $A_3$ in the proof of Proposition~\ref{prop:planar_is_reversible}.  The orientation at each vertex is clockwise with respect to the page.}
\label{fig:a1a2a3}%
\end{figure}

Before moving on, we observe that \emph{any} path from $x$ to $x'$ not passing through $x$ or $x'$ except at its endpoints will witness that one of the cycles $C_i$ is nonseparating, so by Lemma~\ref{lemma:almostseparating} any such path has length 1.
Indeed, each edge at $x$ other than the $A_i$ lies on the right side of exactly one cycle $C_i$, and on the left side of the other two; similarly, each edge at $x'$ lies on the left of exactly one $C_i$ and on the right of the other two. Thus the pigeonhole principle implies that no path from $x$ to $x'$ can lie on the same side of $C_i$ for all three simultaneously.

We may also assume that no edge $e$ connects $x$ to $x'$ and lies between $A_2$ and $A_3$ at $x$, after repeating the following reduction step.  Suppose $e$ is such an edge.  If $e$ lies between $\ov A_1$ and $\ov A_2$ at $x'$, then replace $A_2$ with $e$.  If $e$ lies between $\ov A_1$ and $\ov A_3$ at $x'$, then replace $A_3$ with $e$.  If $e$ lies between $\ov A_2$ and $\ov A_3$ at $x'$, then replace either $A_2$ or $A_3$ with $e$.  In each case, this reduces the number of edges at $x$ between $A_2$ and $A_3$ while preserving the cyclic orderings of $A_1,A_2,A_3$ and $\ov A_1,\ov A_2,\ov A_3$ at $x$ and $x'$, so we may repeat this step until there are no edges to $x'$ lying between $A_2$ and $A_3$ at $x$.

Let $(\rho,x)$ be any unicycle with $C(\rho)=C_3$ (e.g.\ by adding $A_1$ to a spanning tree that uses $A_2$; see \S\ref{sec:spanningtrees} for details). We have $\rho[x]=A_1$ and $\rho[x']=\ov{A}_2$ by definition. Define $\sigma$ by $(\ov{\rho},x)\leadsto_{x'}(\sigma,x')$. We had $\ov{\rho}[x]=A_2$ and $\ov{\rho}[x']=\ov{A}_1$, and the latter implies $\sigma[x']=\ov{A}_1$. But we know moreover that $\sigma[x]=A_3$.  This is because in the rotor routing process $(\ov \rho, x) \leadsto (\sigma, x')$, starting from the last time the chip leaves $x$ and ending at its arrival at $x'$, the chip traces out a path from $x$ to $x'$.  That path has length 1 by the argument above, which means that the chip arrives for the first time at $x'$ along some edge from $x$.  But we have assumed that $A_3$ is the first such edge that the rotor at $x$ encounters when it starts at $A_2$. 

Finally, define $\tau$ by $(\ov{\sigma},x')\leadsto_x(\tau,x)$. We had  $\ov{\sigma}[x]=A_1$ and $\ov{\sigma}[x']=\ov{A_3}$, and the former implies $\tau[x]=A_1$.
We have assumed that all cycles are reversible, so $(\rho,x)\corr(\ov{\rho},x)$ and $(\sigma,x')\corr (\ov{\sigma},x')$. By transitivity this implies $(\rho,x)\corr (\tau,x)$; since $\tau[x]=\rho[x]=A_1$, it must be that $\tau=\rho$. However, this would mean that the rotor $\tau[x']$ ends up at $\ov{A}_2$, which is impossible: this rotor began at $\ov{\sigma}[x']=\ov{A}_3$, so the first edge to $x$ that it encounters cannot be $\ov{A}_2$ (it would encounter $\ov A_1$ first, if not some other edge). This contradiction completes the proof that not all cycles in $G$ can be reversible if $G$ has nonseparating cycles.
\end{proof}


\section{Spanning trees and basepoint-independence} \label{sec:spanningtrees}


\para{Spanning trees and unicycles}
The reason that we studied unicycles so carefully in the previous section is that they are closely related to spanning trees, as we now explain.  Let $T\in \T(G)$ be a spanning tree.  Any two vertices $x$ and $y$ are connected by a unique geodesic path in $T$; we write $\gamma_T(x,y)$ for that path, oriented from $x$ to $y$. Therefore if $e$ is an edge of $G$ connecting $y$ to $x$, the undirected graph $T\sqcup e$ contains a unique cycle. 

Now given a directed edge $e$ from $y$ to $x$ and a spanning tree $T$, define a rotor configuration $\rho_e(T)$ as follows.  Let $T_y$ denote the collection of directed edges obtained from $T$ by directing each edge towards the ``root'' $y$. Then the collection of directed edges $T_y\sqcup e$ forms a rotor configuration $\rho_e(T)$ with unique directed cycle $\gamma_T(x,y)\sqcup e$.  We will denote this directed cycle by $C_e(T)$. If $T$ already contains $e$, then $C_e(T)$ is the length 2 cycle consisting only of $e\sqcup \ov{e}$; note that this cycle is trivially reversible in any ribbon graph. 

If we took instead the edge $\ov{e}$ from $x$ to $y$, the rotor configuration $\rho_{\ov{e}}(T)$ comes from the directed graph $T_x\sqcup \ov{e}$, with directed cycle $C_{\ov{e}}(T)=\gamma_T(y,x)\sqcup \ov{e}$. The following observation  will be fundamental for us.
\begin{lemma}
\label{lemma:ovrho}We have
\[
\ov{\rho_e(T)}=\rho_{\ov{e}}(T).\mathqedhereforthm
\]
\end{lemma}
\begin{proof}
For $s\in \gamma_T(x,y)$ this follows because 
\[
C_{\ov{e}}(T)=\gamma_T(y,x)\sqcup \ov{e} =\ov{\gamma_T(x,y)\sqcup e}=\ov{C_e(T)}.\]
For vertices $z\not\in \gamma_T(x,y)$, we have $T_x[z]=T_y[z]$, since the geodesics $\gamma_T(z,x)$ and $\gamma_T(z,y)$ share an initial segment. Then $\rho_e(T)[z]=\rho_{\ov e}(T)[z]$, and $\rho_e(T)[z] = \ov{\rho_e(T)}[z]$ since $z$ is not on the cycle.
\end{proof}


Given a spanning tree $T\in \T$ and an edge $e$ from $y$ to $x$, the rotor configuration $(\rho_e(T),y)$ is a unicycle since $y\in C_e(T)$.
Conversely, given a unicycle $(\rho,v)$, deleting the directed edge $\rho[v]$ yields a spanning tree $T_v(\rho)$ with all edges directed towards the root $v$.
This correspondence between spanning trees and rotor configurations lets us partially model  the action of the sandpile group on spanning trees.

\begin{proposition}[\textbf{Rotor-routing process computes sandpile action}]
\label{prop:rotorscompute}
Given a spanning tree $T$ in a ribbon graph $G$, let $e$ be an edge not in $T$, and $s$ any vertex on the $T$-geodesic between the endpoints $y,x$ of $e$ (including $x$ or $y$). Let $\rho=\rho_e(T)$ (so that $T = T_y(\rho)$). Then $(\rho, s)$ is a unicycle, and if $\sigma$ is the rotor configuration defined by $(\rho,s)\leadsto_y (\sigma,y)$, then
\[(s-y)_y(T)=T_y(\sigma).\mathqedhereforthm\]
\end{proposition}
\begin{proof}
The rotor configuration $\rho$ is equal to $T_y\sqcup e$, where $T_y$ is obtained by directing each edge of $T$ towards $y$. The configuration $\rho$ then has a unique cycle, namely $\gamma_{T}(x,y) \sqcup e$, and by assumption $s$ is on that cycle, so $(\rho, s)$ is a unicycle. Next, by definition, $(s-y)_y(T)$ is obtained by beginning with $T_y$, placing the chip at $s$, and running the rotor-routing process until the chip arrives at $y$ for the first time. The presence of the rotor $\rho[y]=e$ does not affect this process, since $y$ is not reached until the last step. Therefore the rotor-routing process $(\rho,s)\leadsto (\sigma,y)$ takes precisely the same steps as the process $T\leadsto (s-y)_y(T)$. We conclude that $\sigma=(s-y)_y(T)\sqcup e$. Since $\sigma=T_y(\sigma)\sqcup e$ by definition, this proves the lemma.
\end{proof}
\pagebreak

\para{Basepoint-independence and reversible cycles}
Before giving our last few lemmas and proving our main theorem, let us give some intuition as to why reversibility should be at all related to basepoint-independence. Consider a spanning tree $T_0$, and let $x$ and $y$ be adjacent vertices connected by an edge $e$ from $x$ to $y$. Suppose that we wish to compute $D'_y(D_x(T_0))$ for some divisors $D,D'\in \Div^0(G)$. The first step is to compute $T\coloneq D_x(T_0)$. By Proposition~\ref{prop:rotorscompute}, the computation of $D_x(T_0)$ can be modeled (at least for some divisors) by a rotor-routing process; this process terminates with the rotor configuration $T_x$, which is the spanning tree $T$ oriented towards the basepoint $x$.

The next step is to compute $D'_y(T)$, and to do this we first need to re-orient the tree $T$ towards the basepoint $y$. This does not affect any edges except those on the $T$-geodesic between $x$ and $y$, which are reversed. If we add the edge $e$ to our rotor configuration $T_x$ to obtain $\rho_e(T)=T\sqcup e$, the reversal of the $T$-geodesic $\gamma_T(x,y)$ amounts to the reversal of the cycle $\gamma_T(x,y)\sqcup e$ in the rotor configuration $\rho_e(T)$. If this cycle is reversible, then we can compute both steps $T_0\mapsto T\coloneq D_x(T_0)$ and $T\mapsto D_y(T)$ as part of a \emph{single rotor-routing process}, even though these actions involve two \emph{different} basepoints. This basic  observation is the heart of the relation between reversibility and basepoint-independence, upon which our main theorem rests. The next lemma capitalizes on this observation.

\begin{lemma}
\label{lemma:ongeodesic}
Given a spanning tree  $T$ in a ribbon graph $G$ and an edge $e$ from $y$ to $x$, let $s$ be any vertex on the  geodesic $\gamma_T(x,y)$ (including $s=x$ or $s=y$), and set $T'=(s-y)_y(T)$.
\begin{enumerate}[(a)]
\item If $C_e(T)$ and $C_e(T')$ are both reversible, then \begin{equation}
\label{eq:ongeodesic} (y-x)_x\big((s-y)_y(T)\big)=(s-x)_x(T).
\end{equation}
\item If \eqref{eq:ongeodesic} holds, then either $C_e(T)$ and $C_e(T')$ are both reversible, or $C_e(T)$ and $C_e(T')$ are both non-reversible.\qedhere
\end{enumerate}
\end{lemma}

\begin{proof} Let $\rho=\rho_e(T)$; as in Proposition~\ref{prop:rotorscompute}, $(\rho, s)$ is a unicycle. Define rotor configurations $\psi$, $\sigma$, and $\tau$ by
\[\begin{array}{rclcl}
(\overline{\rho},s)&&&\leadsto_x&(\psi,x)\\
(\rho,s)&\leadsto_y&(\sigma,y)\\
&&(\overline{\sigma},y)&\leadsto_x&(\tau,x)
\end{array}\]
Note that $\psi[x]=\ov{\rho}[x]=\ov{e}$, and similarly $\sigma[y]=\rho[y]=e$ implies that $\tau[x]=\ov{\sigma}[x]=\ov{e}$.
In fact, 
by Lemma~\ref{lemma:ovrho} and Proposition~\ref{prop:rotorscompute}, we have
\begin{itemize}
\item $\psi=(s-x)_x(T) \sqcup \ov{e}$
\item $\sigma= \qquad\qquad(s-y)_y(T) \sqcup e$
\item $\tau=(y-x)_x((s-y)_y(T)) \sqcup \ov{e}$
\end{itemize}
Thus we see that equality \eqref{eq:ongeodesic} is exactly the condition $\tau=\psi$. Moreover, since $\psi[x]=\tau[x]=\ov{e}$, Lemma~\ref{lemma:therotorknows} implies that this is equivalent to $(\psi,x)\corr (\tau,x)$.
Consider the following diagram relating these six rotor configurations:
\[\xymatrix{
  &(\rho,s)\ar@{2{<}~2{>}}[rr]\ar_{\txt{\large ?}}@{2{<}~2{>}}[dl]&&(\sigma,y)\ar^{\txt{\large ?}}@{2{<}~2{>}}[dr]\\
  (\ov{\rho},s)\ar@{2{<}~2{>}}[dr]&&&&(\ov{\sigma},y)\ar@{2{<}~2{>}}[dl]\\
  &(\psi,x)\ar_{\txt{\large ?}}@{2{<}~2{>}}[rr]&&(\tau,x)
}\] Three of these edges hold by definition, namely $(\rho,s)\corr (\sigma,y)$, $(\overline{\sigma},y)\corr (\tau,x)$, and 
$(\overline{\rho},s)\corr (\psi,x)$. By the results of Section~\ref{sec:reversibility}, the upper left edge $(\rho,s)\corr (\ov{\rho},s)$ is equivalent to the cycle $C(\rho)=C_e(T)$ being reversible. Similarly the upper right edge is equivalent to $C(\sigma)=C_e(T')$ being reversible. Finally, we have already shown that the bottom edge is equivalent to the assertion \eqref{eq:ongeodesic}.
Therefore parts (a) and (b) of the lemma each assert that if two of the remaining three edges hold, the third does as well. But this follows immediately from the transitivity of this equivalence relation.\end{proof}

\begin{corollary}
\label{corollary:first}
Let $x$ and $y$ be adjacent vertices of a planar ribbon graph $G$, and let $T$ be a spanning tree. If $s\in \gamma_T(x,y)$, then 
$(s-y)_x(T)=(s-y)_y(T)$.
\end{corollary}
\begin{proof} By Lemma~\ref{lemma:almostseparating}, the hypothesis of Lemma~\ref{lemma:ongeodesic}(a) is always satisfied if $G$ is planar, so 
$(y-x)_x\big((s-y)_y(T)\big)=(s-x)_x(T).$
Now applying $(x-y)_x$ to both sides, we conclude that $(s-y)_y(T) = (s-y)_x(T)$.
\end{proof}
We emphasize that we are \emph{not} yet claiming that 
$(s-y)_x=(s-y)_y$, 
since the condition of Corollary~\ref{corollary:first} that $s$ lies on the $T$-geodesic from $x$ to $y$ need not hold for all trees $T$. 

\begin{lemma}
\label{lemma:rununtil}
Let $T\in\T$ be a spanning tree of an arbitrary ribbon graph $G$. Let $x$, $z$, and $s$ be arbitrary 
vertices such that $s$ is the first vertex on $\gamma_T(s,x)$ visited by the rotor-routing process starting with $(T_x,z)$. Then \[(z-s)_x(T)=(z-s)_s(T).\mathqedhereforthm\]
\end{lemma}
\begin{proof}
Define $T'_s$ by $(T_s,z)\leadsto_s(T'_s,s)$; by definition, the underlying spanning tree $T'$ of $T'_s$ is equal to $(z-s)_s(T)$. Similarly, define $T''_x$ by $(T_x,z)\leadsto_x(T''_x,x)$, with underlying spanning tree $T''=(z-x)_x(T)$.

As we noted in the proof of Lemma~\ref{lemma:ovrho}, the configurations $T_x$ and $T_s$ coincide for all vertices $w\not\in\gamma_T(s,x)$, while on this geodesic, the difference between $T_x$ and $T_s$ is that  $\gamma_T(s,x)$ is reversed and thus replaced by $\gamma_T(x,s)$. Therefore if the rotor-routing processes are run in parallel starting with $(T_x,z)$ and $(T_s,z)$, the same steps will be taken until the chip first arrives at $\gamma_T(s,x)$. By our assumption on $s$, this occurs when the chip reaches $s$, i.e.\ when the second process reaches $(T'_s,s)$. The first process coincides except that the geodesic between $x$ and $s$ is  directed towards $x$ instead. But by the first sentence of this paragraph, this is the directed tree $T'_x$ obtained from $T'$ by directing its edges towards $x$. To sum up, we have observed that $(T_x,z)\leadsto_s(T'_x,s)$.

By definition, $(s-x)_x(T')$ can be computed by starting with $(T'_x,s)$ and running the rotor-routing process until the chip reaches $x$.
The above shows that the first process reaches $(T'_x,s)$ as an intermediate stage; therefore  continuing this process eventually yields $(T''_x,x)$, since $(T_x,z)\leadsto_x(T''_x,x)$ by definition. This shows that $(s-x)_x(T')=T''$. Expanding out our notation, this says that
\[(s-x)_x\big((z-s)_s(T)\big)=(z-x)_x(T).\]
Applying $(x-s)_x$ to both sides yields $(z-s)_s(T)=(z-s)_x(T)$, as desired.\end{proof}

\begin{corollary}
\label{corollary:fulladjacent}
Let $T$ be a spanning tree of a planar ribbon graph $G$, and let $x$ and $y$ be adjacent vertices. Then for any vertex $z$ we have \[(z-x)_x(T)=(z-x)_y(T).\mathqedhereforthm\]
\end{corollary}
\begin{proof}
Let $s$ be the first vertex on $\gamma_T(x,y)$ that is reached by the rotor-routing process starting with $(T_x,z)$. Note that $s$ is \emph{a fortiori} the first vertex visited on $\gamma_T(x,s)$ or on $\gamma_T(s,y)$ as well. We use this to apply Lemma~\ref{lemma:rununtil} twice, giving
\begin{equation}
\label{eq:ZS}
(z-s)_x(T)=(z-s)_s(T)=(z-s)_y(T).
\end{equation}
Write $T'$ for the tree $(z-s)_s(T)$ appearing in \eqref{eq:ZS}.  Notice that by definition of $s$, in the rotor routing process that computes $(z-s)_s(T)$, the chip never touches any vertex on $\gamma_T(x,y)$ until it reaches $s$, so $\gamma_{T'}(x,y) = \gamma_T(x,y)$.   So $s$ lies on $\gamma_{T'}(x,y).$  By Corollary~\ref{corollary:first} applied to $T'$ (with the roles of $x$ and $y$ reversed), we have
\[(s-x)_x(T')=(s-x)_y(T').\] The desired result follows once this equation is combined with \eqref{eq:ZS}, by applying $(s-x)_x$ to the left-hand side of \eqref{eq:ZS} and applying $(s-x)_y$ to the right-hand side of \eqref{eq:ZS}.
\end{proof}

We can now prove the main theorem.
\begin{proof}[Proof of Theorem~\ref{theorem:main}]
($\Rightarrow$) First, assume that $G$ is a planar ribbon graph. Our goal is to prove  for any $D\in \Div^0(G)$ that $D_x=D_y$ as elements of $\Aut(\T)$ for arbitrary vertices $x$ and $y$.
Since $G$ is connected, it suffices to prove this when $x$ and $y$ are adjacent, with the general case following by induction.

Fix adjacent vertices $x$ and $y$. The group $\Div^0(G)$ is generated by divisors of the form $z-x$ as $z$ ranges over the vertices of $G$. But we proved in Corollary~\ref{corollary:fulladjacent} that $(z-x)_x=(z-x)_y$ when $x$ and $y$ are adjacent. It follows that $D_x=D_y$ for any $D\in \Div^0(G)$. This completes the proof that if $G$ is a planar ribbon graph, the action of $\Pic^0(G)$ on $\T$ is independent of the basepoint.

($\Leftarrow$) Conversely, assume that the action of $\Pic^0(G)$ on $\T$ is independent of the basepoint.
Choose any edge $e$ of $G$ with endpoints $x$ and $y$. We claim that  for any $T\in \T$ and any $D\in \Div^0(G)$, the cycle $C_e(T)$ is reversible if and only if the cycle $C_e(D(T))$ is reversible. (We have dropped the subscripted basepoints, since by assumption the action of $\Pic^0(G)$ does not depend on them.) We first prove the claim in the case when $D=z-y$ for some $z$.
 
As in the proof of Corollary~\ref{corollary:fulladjacent}, let $s$ be the first vertex on $\gamma_T(x,y)$ reached by the rotor-routing process starting with $(T_x,z)$. Since we have assumed independence of basepoint, this process $(T_x,z)\leadsto_s (T'_s,s)$ computes $T'=(z-s)(T)$. Since this process does not reach  $\gamma_T(x,y)$ until the last step, the cycle $C_e(T)$ remains unchanged, i.e. $C_e(T)=C_e((z-s)(T))$.

We may now apply Lemma~\ref{lemma:ongeodesic} to $(z-s)(T)$, since $s$ lies on the cycle $C_e(T)=C_e((z-s)(T))$. The condition \eqref{eq:ongeodesic} is always satisfied by our assumption that the action is independent of the basepoint, so applying Lemma~\ref{lemma:ongeodesic}(b) implies that \begin{align*}&C_e(T)=C_e((z-s)(T))\text{ is reversible}\\\iff\qquad &C_e\big(((s-y)+(z-s))(T)\big)=C_e((z-y)(T))\text{ is reversible,}
\end{align*} 
completing the proof of the claim for $D=z-y$. Since $\Div^0(G)$ is generated by such elements, the claim follows.

We now show that every cycle $C$ in $G$ is reversible. Given a cycle $C$, let $e$ be any edge of $C$. Extend the edge $e$ to a spanning tree $T$, and separately extend the path $C-e$ to a spanning tree $T'$. 
Since $\Pic^0(G)$ acts transitively on $\T$, there exists $D\in \Div^0(G)$ for which $D(T)=T'$. Thus the claim above shows that $C_e(T)$ is reversible if and only if $C_e(T')$ is reversible. But $C_e(T)$ is the trivial cycle $e\sqcup \ov{e}$, which is always reversible in any ribbon graph, while $C_e(T')$ is our original cycle $C$. We conclude that every cycle in $G$ is reversible. By Proposition~\ref{prop:planar_is_reversible}, $G$ is a planar ribbon graph.
\end{proof}

\para{Acknowledgements} We are very grateful to Jordan Ellenberg for calling our attention to the question considered in this paper, and we thank Math Overflow for providing a venue for the question. The first and third authors would also like to thank the second author for suggesting that the three of us get together to work on this problem, which was great fun!

\begin{tabular}{lll}
Department of Mathematics&
Department of Mathematics&
Department of Computer Science\\
Harvard University&
Stanford University&
University of Toronto\\
&&Sandford Fleming Building\\
One Oxford Street&
450 Serra Mall&
10 King's College Road\\
Cambridge, MA 02138&
Stanford, CA 94305&
Toronto, Ontario M5S 3G4 \\
\href{mailto:mtchan@math.harvard.edu}{\nolinkurl{mtchan@math.harvard.edu}}&
\href{mailto:church@math.stanford.edu}{\nolinkurl{church@math.stanford.edu}}&
\href{mailto:jgrochow@cs.toronto.edu}{\nolinkurl{jgrochow@cs.toronto.edu}}
\end{tabular}

%
%

\end{document}